\documentclass[12pt,a4paper,oneside,openright]{amsart}
\usepackage[left=3.5cm,right=3cm,top=3cm,bottom=3cm]{geometry}

\author{Francesca Angrisani}
\title{Autonomous functionals with asymptotic $(p,q)$-structure.}
\date{}
\usepackage[utf8]{inputenc}
\usepackage{amsmath}
\usepackage{amsfonts}
\usepackage{amssymb}
\usepackage{amsthm}
\usepackage{esint}
\usepackage{graphicx}
\usepackage[english]{babel}
\usepackage[T1]{fontenc}

\usepackage{graphicx}
\usepackage{xcolor}
\graphicspath{{./Images/}}

\newtheorem{thm}{Theorem}[]
\newtheorem{cor}[thm]{Corollary}
\newtheorem{lem}[thm]{Lemma}
\newtheorem{prop}[thm]{Proposition}

\theoremstyle{definition}
\newtheorem*{mydef}{Definition}

\theoremstyle{remark}

\theoremstyle{remark}
\newtheorem{rmk}{Remark}
\usepackage{setspace}
\DeclareMathOperator*{\Reg}{\text{Reg}}

\def\Xint#1{\mathchoice 
	{\XXint\displaystyle\textstyle{#1}}%
	{\XXint\textstyle\scriptstyle{#1}}%
	{\XXint\scriptstyle\scriptscriptstyle{#1}}%
	{\XXint\scriptscriptstyle\scriptscriptstyle{#1}}%
	\!\int} 
\def\XXint#1#2#3{{\setbox0=\hbox{$#1{#2#3}{\int}$} 
		\vcenter{\hbox{$#2#3$}}\kern-.5\wd0}} 
\def\fint{\Xint -}

\begin{document}
\begin{abstract}
	We obtain local Lipschitz regularity for minima of autonomous integrals in the calculus of variations, assuming $q$-growth hypothesis and $W^{1,p}$-quasiconvexity only asymptotically, both in the sub-quadratic and the super-quadratic case. 
\end{abstract}

	\maketitle 
	\medskip
	\medskip
	\section{Introduction}
	In this paper we study variational integrals of the type
	$$\mathcal{F}(u)=\int_\Omega f(Du(x))\,dx \quad \text{ for } u:\Omega \to \mathbb{R}^N$$
	where $\Omega$ is an open bounded set in $\mathbb{R}^n$, $n\ge 2$, $N\ge 1$. Here, the Lagrangian function $f: \mathbb{R}^{Nn} \to \mathbb{R}$ is a continuous function satisfying the following properties:
	\begin{itemize}
		\item[(A.1)] \textbf{Regularity}- $f \in C^2(\mathbb{R}^{nN},\mathbb{R})$
		\item[(A.2)] \textbf{$q$-Growth}- $|f(z)|\le \Gamma(1+|z|)^q$
		\item[(A.3)] \textbf{Asymptotical strict $W^{1,p}$-quasiconvexity}- There exists $M>>0$, $\gamma>0$ and a continuous function $g \in W^{1,p}$ such that $$f(z)=g(z),\quad \forall z:\, |z|>M$$ and such that $g$ is strictly $W^{1,p}$-quasiconvex, i.e. satisfies $$\fint_{B_1} g(z+D\varphi)\ge g(z)+\gamma\fint_{B_1}(1+|D\varphi|^2)^{\frac{p}{2}-1}|D\varphi|^{2},\quad \forall z,\, \forall \varphi \in C_0^\infty(B_1,\mathbb{R}^N)$$  
	\end{itemize}
where in this text, $p$ and $q$ will always denote real numbers that satisfy the inequalities $$1< p \le q<p+\frac{\min\{2,p\}}{2n}$$.\\
We will study local $W^{1,p}$-minimizers of $\mathcal{F}$ i.e. functions $u \in W^{1,p}(\Omega,\mathbb{R}^N)$ such that
$$\mathcal{F}(u+v)\ge \mathcal{F}(u) \quad \forall v \in W_0^{1,p}(\Omega,\mathbb{R}^N)$$
and, in the following, we will simply refer to them as "minimizers".\\
In \cite{Schmidt}, T. Schmidt proved that if $M=0$, $u$ is $C^{1,\alpha}$ in an open dense subset of $\Omega$.
We will prove the following result:
\begin{thm} \label{maintheorem}
	Let $f$ satisfy hypotheses $(A.1)$,$(A.2)$ and $(A.3)$ and let $u$ be a local minimizer of the corresponding functional $\mathcal{F}$. Let $z_0\in \mathbb{R}^{nN}$ such that $|z_0|>M+1$ and assume there is a $x_0 \in \mathbb{R}^n$ with the property that $$\fint_{B_\rho(x_0)}|Du-z_0|^p\to 0\quad\text{as}\quad \rho\to 0^+,$$
	then $x_0 \in \Reg(u)$, where $\Reg(u)=\{x \in \Omega: \, u \text{ is Lipschitz in a neighbourhood of } x\}$.\\
	Moreover, $\Reg(u)$ is a dense open subset of $\Omega$
\end{thm}
\section{Technical lemmas and definitions}
\begin{lem} \label{lemmaasintotico}
If $f$ is locally bounded from below, then the function $g$ in $(A.3)$ can be chosen such that $g\le f$.\\
Moreover, in this case, assuming $(A.3)$ is equivalent to assuming the existence of a positive constant $M>0$ big enough such that the following holds: 
\begin{multline}\tag{A.3'}\label{treppiufforte}
\fint_{B_1} f(z+D\varphi)\ge f(z)+\gamma\fint_{B_1}(1+|D\varphi|^2)^{\frac{p}{2}-1}|D\varphi|^{2}\,\, ,\\  \quad \forall z:\, |z|>M,\, \forall \varphi \in C_0^\infty(\Omega,\mathbb{R}^N). \quad
\end{multline}
\end{lem} 
\begin{proof}
	The proof of the first part uses the theory of quasiconvex envelopes and it is identical to what is shown in (\cite{Carozza}, Thm 2.5 (ii)), but we will repeat it for the convenience of the reader.\\
	Let us start with the case $p>2$.\\
	Assume that $f(z)=g(z)$ for all $z: |z|>M$ and that $g$ satisfies strict quasiconvexity for a constant $\gamma>0$, i.e.:
	\begin{equation}\label{quasiconvexitydig}
	\fint_{B_1} g(z+D\varphi(x))\,dx \ge g(z)+\gamma\fint_{B_1}(1+|D\varphi(x)|^2)^{\frac{p}{2}-1}|D\varphi|^2
	\end{equation}\\
	Let $K=\sup\limits_{|z|\le M}(g-f)(z)$ and notice that $K<\infty$ since $f$ is locally bounded from below and $g$ is locally bounded.\\
	Introduce an auxiliary function $h$ which is smooth and non-negative on $\mathbb{R}^{nN}$, with compact support and such that $|D^2h|\le \gamma$ on $\mathbb{R}^{nN}$ and $h(z)\ge K$ for $|z|\le M$.\\
	We now claim that $\tilde{g}=g-h$ is uniformly strictly quasiconvex and satisfying $\tilde{g}\le f$ for all $z \in \mathbb{R}^{nN}$ and $\tilde{g}(z)=f(z)$ for large enough $z$. Of course $\tilde{g}(z)=f(z)$ outside the support of $h(z)$ and $\tilde{g}\le f$ is a trivial consequence of the fact that $h$ is always larger than the difference between $g$ and $f$.\\
	To prove $\tilde{g}$ is strictly quasiconvex, let us consider, for any $\varphi \in C_0^\infty(\Omega,\mathbb{R}^N)$, the quantity:
	\begin{multline*}h(z+D\varphi)-h(z)-Dh(z)D\varphi=\int_0^1 Dh(z+tD\varphi)D\varphi \,dt -Dh(z)D\varphi=\\=|D\varphi|^2\int_0^1\int_0^1 tD^2h(z+stD\varphi)\,ds\,dt\end{multline*}
	so
	$$h(z+D\varphi) \le h(z)+Dh(z)D\varphi+\frac{\gamma}{2}|D\varphi|^2$$
	By passing to the integral average over $B_1$ and changing all signs we obtain
	$$\fint_{B_1} -h(z+D\varphi) \ge -h(z)-\frac{\gamma}{2}\fint_{B_1}|D\varphi|^2.$$
	By summing this inequality with \eqref{quasiconvexitydig} we obtain the quasiconvexity of $\tilde{g}$.\\
	To prove the second statement in the lemma, assume now $g\le f$ and choose $z$ such that $|z|>M$. We have:
	\begin{multline*}\fint_{B_1} f(z+D\varphi)\ge \fint_{B_1} g(z+D\varphi) \ge g(z)+\gamma\fint_{B_1}(1+|D\varphi|^2)^{\frac{p}{2}-1}|D\varphi|^{2} =\\=f(z)+\gamma\fint_{B_1}(1+|D\varphi|^2)^{\frac{p}{2}-1}|D\varphi|^{2}.\end{multline*}
	Now, the only thing we need to change in the case $p\le 2$ is to take a smooth function $h$ such that $|D^2h(z)|\le \gamma (1+|z|^2)^{\frac{p}{2}-1}$. This can easily be done by considering $\tilde{h}(z)=h\left(\frac{z}{d}\right)$ (where $h$ is the same function used in the proof of the case $p>2$) for large enough $d$, i.e. for $d>(1+M^2)^{1-\frac{p}{2}}$
\end{proof}
\begin{cor} \label{corollario}
	Let $f$ satisfy $(A.1)$,$(A.2)$ and $(A.3)$. Then it satisfies $(A.3')$.
\end{cor}
\begin{mydef}[Excess]
Let $\beta>0$ and let us consider 
$$V^\beta(z)=(1+|z|^2)^{\frac{\beta-1}{2}}z \text{ and } W^{\beta}(z)=(1+|z|)^{\beta-1}z$$
with $V$ and $W$ comparable in the sense that there exists a constant $c>0$ depending only on $\beta$ such that for all $z$ we have \begin{equation} \label{comparability} c^{-1}|W^{\beta}(z)|\le |V^{\beta}(z)|\le c|W^{\beta}(z)|.\end{equation}
We will often consider the quantity $\left|V^{\frac{p}{2}}(z)\right|^2$ in our computations. The advantage of sometimes dealing with the equivalent quantity $|W^{\frac{p}{2}}(z)|^2$ is the fact that it can be easily proven that $$z\mapsto |W^{\frac{p}{2}}(z)|^2$$ is convex for all $p\ge 1$.\\
For $u \in W^{1,p}(B_\rho(x_0),\mathbb{R}^N)$ and $z \in \mathbb{R}^{nN}$ define the quantity
$$\Phi_p(u,x_0,\rho,z):=\fint_{B_\rho(z_0)}|V^{\frac{p}{2}}(Du-z)|^2 \text{ and }$$
and in particular, we call
$$\Phi_p(u,x_0,\rho):=\Phi_p\left(u,x_0,\rho,\fint_{B_\rho(x_0)}Du\right).$$
the \textbf{excess function} of the minimizer $u$ in $x_0$
\end{mydef}
In \cite{Schmidt}, T. Schmidt proved that if $u$ is a $W^{1,p}$-minimizer of $\mathcal{F}$ on $B_\rho(x_0)$, for all $L>0$ and $\alpha\in (0,1)$ there exist an $\varepsilon_0>0$ such that if \begin{equation}\label{quella}\Phi_p(u,x_0,\rho)\le \varepsilon_0 \text{ and } \left|\fint_{B_\rho(x_0)} Du\right|\le L\end{equation} then $u \in C^{1,\alpha}_{loc}(B_\rho(x_0);\mathbb{R}^N)$.\\
We will now replicate his reasoning in the weaker hypothesis of asymptotic quasiconvexity.\\
To do so it will be first necessary to observe the following.
\begin{lem} \label{puntiregolari}
	If there exists $z_0$, $|z_0|>M+1$ and $x_0$ such that:$$\fint_{B_\rho(x_0)} \left|V^{\frac{p}{2}}\left(Du-z_0\right)\right|^2 \to 0 \text{ as } \rho \to 0^+$$
	then there exists $r_1=r_1(x_0,z_0)$such that for all $r<r_1$
	$$\left|\fint_{B_r(x_0)} Du\right|>M+1.$$
\end{lem}
\begin{proof}
	Let $|z_0|=M+1+\varepsilon$. Then by definition of limit there must be a $r_1$ (of course this depends on the specific values of $x_0$ and $z_0$) such that for all $r<r_1$ we have:
	$$\fint_{B_r(x_0)} \left|W^{\frac{p}{2}}(Du-z_0)\right|^2 \le \left|W^{\frac{p}{2}}\left(\frac{\varepsilon}{2}\right)\right|^2, \quad \forall r<r_1$$
	where we have also used inequality \eqref{comparability}, comparing $V$ and $W$.\\
	which, by Jensen inequality means:
	$$\left|\fint_{B_r(x_0)} Du-z_0\right|\le \frac{\varepsilon}{2},\quad \forall r<r_1$$
	which gives:
	$$\left|\fint_{B_r(x_0)} Du\right|\ge |z_0|-\frac{\varepsilon}{2}=M+1+\varepsilon-\frac{\varepsilon}{2}>M+1, \quad \forall r<r_1.$$
\end{proof}
We start with a lemma by P. Marcellini (Step 2 of Thm 2.1 in \cite{Marcellini}).
\begin{lem}
Let $f$ satisfy assumptions $(A.1)$, $(A.2)$ and $(A.3)$ and $z$ be such that $|z|>M$. Then $|Df(z)|\le \Gamma_2(1+|z|^{q-1})$.
\begin{proof}
Let $z$ be such that $|z|>M$.
As shown in (\cite{Giusti}, proposition 5.2), quasiconvexity in a given $z$ implies rank-one convexity in $z$. This implies that the modulus of each partial derivative is bounded by the modulus of the difference quotient, which can be bounded by the $(q-1)$-th power of the argument by the inequality in hypothesis.\\
In symbols, if $\varphi_i(\cdot)=f(\zeta_1,\ldots,\zeta_{i-1},\cdot,\zeta_{i+1},\ldots,\zeta_{nN})$:
$$|\varphi_i'(\zeta_i)|\le \left|\frac{\varphi(\zeta_i+|\zeta|+1)-\varphi(\zeta_i)}{|\zeta|+1}\right|\le \Gamma_2(1+|\zeta|^{q-1})$$
\end{proof}
\end{lem}
We proceed adapting a result from Acerbi and Fusco (\cite{Acerbi}, Lemma II.3)
\begin{lem} \label{Acerbo}
	Choose any $L>0$. Let $f$ satisfy $(A.1)$, $(A.2)$ and be such that $|Df(z)|\le \Gamma_2(1+|z|^{q-1})$ holds true for all $z:\, |z|>M$, then 
	\begin{itemize}
		\item $|f(z+\eta)-f(z)-Df(z)\eta|\le c_1|V^{q/2}(\eta)|^2$
		\item $|Df(z+\eta)-Df(z)|\le c_2|V^{q-1}(\eta)|$
	\end{itemize}
	for all $\eta \in \mathbb{R}^{nN}$ and for all $z$ such that $M<|z|\le L$, with $c_1$ and $c_2$ depending only on $f$,$n$,$N$,$L$,$\Gamma$,$\Gamma_2$ and $M$.
	\begin{proof}
	Choose $L>0$ and $z$ such that $|z|\le L$ for some $L\in\mathbb{R}$. We start proving the first of the two inequalities. If $|\eta|\le 1$ we have:
	$$|f(z+\eta)-f(z)-Df(z)\eta| \le |D^2f(z+\theta\eta)||\eta|^2\le \max\limits_{|z|\le L+1}|D^2f(z)| |\eta|^2$$.\\
	If $\eta >1$ we start from:
	$$|f(z+\eta)-f(z)-Df(z)\eta|=|Df(z+\theta\eta)\eta-Df(z)\eta|$$
	and then, if $|z+\theta\eta|\le M$, we end with
	$$|Df(z+\theta\eta)\eta-Df(z)\eta|\le 2 \max\limits_{|z|\le \max\{L,M\}} |Df(z)||\eta|$$ otherwise, if $|z+\theta\eta|>M$, we conclude with:
	$$|Df(z+\theta\eta)\eta-Df(z)\eta| \le \Gamma_2(1+|z+\theta\eta|^{q-1})|\eta|+\max\limits_{|z|\le L} |Df(z)||\eta|\le c_1|\eta|^q.$$
	Similar methods are used in the proof of the second inequality.\\
	If $|\eta| \le 1$ we have:
	$$|Df(z+\eta)-Df(z)|=|D^2f(z+\theta\eta)||\eta|\le\max\limits_{|z|\le L+1}|D^2f(z)||\eta|.$$
	If $|\eta|> 1$ and $|z+\eta|\le M$ then:
	$$|Df(z+\eta)-Df(z)|\le 2 \max\limits_{|z|\le \max\{L,M\}} |Df(z)|\le c_2|\eta| $$
	while if $|\eta|>1$ and $|z+\eta|> M$ then:
	$$|Df(z+\eta)-Df(z)|\le \Gamma_2(1+|z+\eta|^{q-1})+\max\limits_{|z|\le L} |Df(z)|\le c_2|\eta|^{q-1}.$$
	
\end{proof}
\end{lem}
\section{Caccioppoli estimate}
Next step is to obtain a Caccioppoli estimate adapting a proof by T. Schmidt (see \cite{Schmidt}, Lemma 7.3). 
To do so we need a few lemmas. The proofs can be found in \cite{Fonseca} and in \cite{Schmidt}.
\begin{lem}\label{Fonseca1}
	Let $0<r<s$ and $B_s\subset \Omega$. We define a bounded linear smoothing operator $$T_{r,s}:W^{1,1}(\Omega;\mathbb{R}^N)\to W^{1,1}(\Omega;\mathbb{R}^N)$$ for $u \in W^{1,1}(\Omega;\mathbb{R}^N)$ and $x \in \Omega$ by
	$$T_{r,s}u(x):=\fint_{B_1}u(x+\theta(x)y)\,dy\text{ where } \theta(x):=\frac{1}{2}\max\{\min\{|x|-r,s-|x|\},0\}.$$
	With this definition, for all $1\le p \le q < \frac{n}{n-1}p$ and all $u \in W^{1,p}(\Omega;\mathbb{R}^N)$ the following assertions are true:
	\begin{enumerate}
		\item $T_{r,s}u \in W^{1,p}(\Omega;\mathbb{R}^N),$
		\item $u=T_{r,s}u$ almost everywhere on $(\Omega\setminus B_s) \cup B_r,$
		\item $T_{r,s}u \in u +W_0^{1,p}(B_s\setminus \overline{B_r};\mathbb{R}^N),$
		\item $|D T_{r,s}u|\le c(n)T_{r,s}|D u|$ almost everywhere in $\Omega$,
		\item $\|T_{r,s}u\|_{L^p(B_s\setminus B_r)} \le c(n,p)\|u\|_{L^p(B_s\setminus B_r)},$
		\item $\|D T_{r,s} u\|_{L^p(B_s\setminus B_r)} \le c(n,p) \|Du\|_{L^p(B_s\setminus B_r)},$
		\item $\|T_{r,s}u\|_{L^q(B_s\setminus B_r)}\le c(n,p,q)(s-r)^{\frac{n}{q}-\frac{n-1}{p}}\left[\sup\limits_{t \in (r,s)} \frac{\tilde{\Xi}(t)-\tilde{\Xi}(r)}{t-r}+\sup\limits_{t \in (r,s)} \frac{\tilde{\Xi}(s)-\tilde{\Xi}(t)}{s-t}\right]^{\frac{1}{p}},$
		\item $\|DT_{r,s}u\|_{L^q(B_s\setminus B_r)}\le c(n,p,q)(s-r)^{\frac{n}{q}-\frac{n-1}{p}}\left[\sup\limits_{t \in (r,s)} \frac{\Xi(t)-\Xi(r)}{t-r}+\sup\limits_{t \in (r,s)} \frac{\Xi(s)-\Xi(t)}{s-t}\right]^{\frac{1}{p}}.$
		\item $\left| V^{\frac{p}{2}}(DT_{r,s}u)\right|^{\frac{2}{p}}\le c T_{r,s}\left[\left|V^{\frac{p}{2}}(Du)\right|^{\frac{2}{p}}\right] \ \forall p:  1\le p\le 2 \ \text{ a.e. in } \Omega, \ c=c(n,p).$
	\end{enumerate}
where we used the abbreviations: $$\tilde{\Xi}(t):=\|u\|^p_{L^p(B_t)}$$ and $$\Xi(t):=\|Du\|^p_{L^p(B_t)}$$.
\end{lem} 
Another lemma that will be useful in obtaining the Caccioppoli estimate is the following, the proof of which can also be found in \cite{Fonseca}.
\begin{lem}\label{Fonseca2}
	Let $-\infty<r<s<+\infty$ and a continuous nondecreasing function $\Xi:[r,s]\to \mathbb{R}$ be given. Then there are $\tilde{r}\in [r,\frac{2r+s}{3}]$ and $\tilde{s}\in [\frac{r+2s}{3},s]$, for which hold:
	$$\frac{\Xi(t)-\Xi(\tilde{r})}{t-\tilde{r}}\le 3\frac{\Xi(s)-\Xi(r)}{s-r}$$ and $$\frac{\Xi(\tilde{s})-\Xi(t)}{\tilde{s}-t}\le 3\frac{\Xi(s)-\Xi(r)}{s-r}$$ for every $t \in (\tilde{r},\tilde{s})$.\\
	In particular, we have $\frac{s-r}{3}\le \tilde{s}-\tilde{r}\le s-r.$
\end{lem}
Now we can prove the Caccioppoli estimate.
\begin{lem}[Caccioppoli Inequality] \label{caccioppolilemma}
	Let $f$ satisfy $(A.1)$,$(A.2)$ and $(A.3)$ for a given $M$. Choose any positive constant $L>M>0$ and a consider $W^{1,p}$-minimizer $u \in W^{1,p}(B_{\rho}(x_0);\mathbb{R}^N)$ of $\mathcal{F}$ on $B_\rho(x_0)$. Then, for all $\zeta \in \mathbb{R}^N$ and $z \in \mathbb{R}^{nN}$ with $M<|z|<L+1$, we have: \begin{equation}\label{caccioppoli}\Phi_p\left(u,x_0,\frac{\rho}{2},z\right)\le c\left[h\left(\fint_{B_\rho(x_0)}\left|V^{\frac{p}{2}}\left(\frac{v}{\rho}\right)\right|^2\,dx\right)+(\Phi_p(u,x_0,\rho,z))^{\frac{q}{p}}\right]\end{equation}
	where we have set $h(t):=t+t^{\frac{q}{p}}$ and $v(x)=u(x)-\zeta-z(x-x_0)$ and where $c$ denotes a positive constant depending only on $n$,$N$,$p$,$q$,$\Gamma$,$L$,$M$,$\gamma$ and $\Lambda_L:=\sup\limits_{|z|\le L+2} |Df^2(z)|$.
\end{lem}
\begin{proof}
	Assume for simplicity $x_0=0$ and choose $$\frac{\rho}{2}\le r < s \le \rho.$$ Define $$\Xi(t):=\int_{B_t}\left[|Dv|^p+\left|\frac{v}{s-r}\right|^p\right]\,dx.$$
	We choose in addition $r\le \tilde{r}< \tilde{s}\le s$ as in Lemma \ref{Fonseca2}. Let $\eta$ denote a smooth cut-off functions with support in $B_{\tilde{s}}$ satisfying $\eta\equiv1$ in $\overline{B_{\tilde{r}}}$ and $0\le \eta \le 1$, $|\nabla\eta|\le \frac{2}{\tilde{s}-\tilde{r}}$ on $B_\rho$. Using the operator from Lemma \ref{Fonseca1}, we set $$\psi:=T_{\tilde{r},\tilde{s}}[(1-\eta)v]\text{ and } \varphi:=v-\psi.$$
	Using properties $(2)$ and $(3)$ from lemma \ref{Fonseca1}, we have $\varphi \in W_0^{1,p}(B_{\tilde{s}};\mathbb{R}^N)$ and $\varphi=v$ on $B_{\tilde{r}}$. Furthermore, we see $$Du-z=Dv=D\varphi+D\psi\text{ on } B_\rho.$$ Using $(A.3')$ (see Corollary \ref{corollario}), from lemma \ref{lemmaasintotico} we obtain that, for every $z$ such that $L+1>|z|>M$ \begin{multline*}\gamma \int_{B_{\tilde{r}}}|V^{\frac{p}{2}}(Dv)|^2\,dx = \gamma \int_{B_{\tilde{r}}}|V^{\frac{p}{2}}(D\varphi)|^2\,dx   \le  \gamma \int_{B_{\tilde{s}}}|V^{\frac{p}{2}}(D\varphi)|^2\,dx  =  \\ = \gamma \int_{B_{\tilde{s}}}\left(1+|D\varphi|^2\right)^{\frac{p}{2}-1}|D\varphi|^2\,dx  \le \int_{B_{\tilde{s}}}[f(z+D\varphi)-f(z)]\,dx =\\= \int_{B_{\tilde{s}}}[f(Du-D\psi)-f(Du)]\,dx +\int_{B_{\tilde{s}}}[f(Du)-f(Du-D\psi)]\,dx +\\ +\int_{B_{\tilde{s}}} [f(z+D\psi)-f(z)]\,dx.\end{multline*}
	Applying the minimality of $u$ and lemma \ref{Acerbo} and adding and subtracting $Df(z)D\psi(x)$ we conclude that $\forall z: \,L+1>|z|>M$,
	\begin{multline} \label{puntodisplit} \gamma \int_{B_r} |V^{\frac{p}{2}}(Dv)|^2\,dx \le \\ \le \int\limits_{B_{\tilde{s}}} \left[\int_0^1 \left( Df(z)-Df(Du-\tau D\psi)\right)\,d\tau D\psi+f(z+D\psi)-f(z)-Df(z)D\psi\right]\,dx \\\le c\int_{B_{\tilde{s}}}\left[\int_0^1 |V^{q-1}(Dv-\tau D\psi)|\,d\tau|D\psi|+|V^{\frac{q}{2}}(D\psi)|^2\right]\,dx .\end{multline}
	Starting from now, we divide the proof in two cases, beginning from the case $p>2$.\\
	Setting $R:=B_{\tilde{s}}\setminus B_{\tilde{r}}$, recalling $\psi\equiv 0$ on $B_{\tilde{r}}$ and some elementary properties of $V$, i.e. \begin{equation}\label{prop1}|V^\beta(A+B)|\le c[|V^\beta(A)|+|V^\beta(B)|]\end{equation} and \begin{equation}\label{prop2} \min\{t^2,t^p\}|V^{\frac{p}{2}}(A)|^2\le |V^{\frac{p}{2}}(tA)|^2\le \max\{t^2,t^p\}|V^{\frac{p}{2}}(A)|^2\end{equation} (see \cite{Schmidt}, Definition 6.1), we infer: \begin{equation}\label{gelato}\int_{B_r} |V^{\frac{p}{2}}(Dv)|^2\,dx\le c\left[\int_R|V^{\frac{q}{2}}(D\psi)|^2\,dx+\int_R|V^{q-1}(Dv)||D\psi|\,dx\right]=:c[I_1+I_2]\end{equation}
	Let us introduce the abbreviation $$\Delta:=\int_{B_s\setminus B_r}\left[\left|V^{\frac{p}{2}}(Dv)\right|^2+\left|V^{\frac{p}{2}}\left(\frac{v}{s-r}\right)\right|^2\right]\,dx.$$
	Using properties $(6)$ and $(8)$ of lemma \ref{Fonseca1} ($q< \frac{np}{n-1}$) and lemma \ref{Fonseca2} we get: 
	\begin{multline} \label{pollo}
	I_1 \le c \left[ \int_R |D\psi|^2\,dx+\int_R |D\psi|^q\,dx\right]  \le \\ \le c \left[\int_R |D[(1-\eta)v]|^2\,dx+\int_R |DT_{\tilde{r},\tilde{s}}\left[(1-\eta)v\right]|^q\,dx\right] \\ \le c \Bigg[\Delta+(s-r)^n\left((s-r)^{1-n}\sup\limits_{t \in (\tilde{r},\tilde{s})}\frac{\Xi(t)-\Xi(\tilde{r})}{t-\tilde{r}}\right)^{\frac{q}{p}}+\\+(s-r)^n\left((s-r)^{1-n} \sup\limits_{t \in (\tilde{r},\tilde{s})}\frac{\Xi(\tilde{s})-\Xi(t)}{\tilde{s}-t}\right)^{\frac{q}{p}}\Bigg]\le \\ \le c\left[\Delta+(s-r)^{n}\left(\frac{\Delta}{(s-r)^n}\right)^{\frac{q}{p}}\right].
	 \end{multline}
	 Using $q<p+\frac{1}{n}<p+1$ and H\"older's inequality we can treat $I_2$ in a similar fashion:
	 \begin{multline} \label{architetto}
	 I_2\le c\int_R (|Dv||D\psi|+|Dv|^{q-1}|D\psi|)\,dx\le \\ \le c\Big[\Big(\int_R |Dv|^2\,dx\Big)^{\frac{1}{2}}\Big(\int_R|D\psi|^2\,dx\Big)^{\frac{1}{2}}+\\+\Big(\int_R|Dv|^p\,dx\Big)^{\frac{q-1}{p}}\Big(\int_R |D\psi|^{\frac{p}{p+1-q}}\,dx\Big)^{\frac{p+1-q}{p}}\Big]\le \\ \le c\Big[\Delta+(s-r)^n\Big(\frac{\Delta}{(s-r)^{n}}\Big)^{\frac{q}{p}}\Big].	 \end{multline}
	 For the last inequality, notice that $\int_R|Dv|^2\,dx$ and $\int_R |Dv|^p\,dx$ are obviously less than $\int_{B_s\setminus B_r}|V^{\frac{p}{2}}(D\psi)|^2$, which is less than $\Delta$, and that $\int_R |D\psi|^2\,dx<\Delta$.\\
	 
	 Combining \eqref{gelato}, \eqref{pollo} and \eqref{architetto} we arrive at
	 $$\int_{B_r} |V^{\frac{p}{2}}(Dv)|^2\,dx\le C_1 \left[\Delta+(s-r)^n\left(\frac{\Delta}{(s-r)^n}\right)^{\frac{q}{p}}\right],$$ where $C_1$ denotes a positive fixed constant depending on $n,N,p,q,\Gamma,\gamma,L,\Lambda_L$.\\
	 Adding $C_1\int_{B_r} |V^{\frac{p}{2}}(Dv)|^2\,dx$ on both sides and dividing by $1+C_1$ we see:
	 \begin{multline}
	 \int_{B_r} |V^{\frac{p}{2}}(Dv)|^2\,dx\le \frac{C_1}{1+C_1}\int_{B_s}|V^{\frac{p}{2}}(Dv)|^2\,dx+\int_{B_\rho}\left|V^{\frac{p}{2}}\left(\frac{v}{s-r}\right)\right|^2\,dx+\\+(s-r)^n\left(\frac{1}{(s-r)^n}\int_{B_\rho}\left[\left|V^{\frac{p}{2}}(Dv)\right|^{2}+\left|V^{\frac{p}{2}}\left(\frac{v}{s-r}\right)\right|^2\right]\,dx\right)^{\frac{q}{p}}.
	 \end{multline}
	 Using Lemma 6.6 from \cite{Schmidt}, we have:
	 \begin{multline}
	 \int_{B_{\rho/2}} |V^{\frac{p}{2}}(Dv)|^2\,dx \le \\ \le c\left[\fint_{B_\rho} \left|V^{\frac{p}{2}}\left(\frac{v}{\rho}\right)\right|^2\,dx+\left(\fint_{B_\rho}\left[\left|V^{\frac{p}{2}}(Dv)\right|^{2}+\left|V^{\frac{p}{2}}\left(\frac{v}{\rho}\right)\right|^{2}\right]\,dx\right)^{\frac{q}{p}}\right],
	 \end{multline}
	 which proves the claim in the case $p>2$.\\
	 We now approach the proof in the case $p\le 2$ restarting from \eqref{puntodisplit} and using a different argument.\\
	 We use the notations of the previous case except for the following modification:
	 $$\Xi(t):=\int_{B_t}\left[\left|V^{\frac{p}{2}}(Dv)\right|^2+\left|V^{\frac{p}{2}}\left(\frac{v}{s-r}\right)\right|^2\right]\,dx.$$\\
	 Exactly as before, we reach
	 $$\int_{B_r} |V^{\frac{p}{2}}(Dv)|^2\,dx \le c\int_{B_{\tilde{s}}}\left[\int_0^1 |V^{q-1}(Dv-\tau D\psi)|\,d\tau|D\psi|+|V^{\frac{q}{2}}(D\psi)|^2\right]\,dx.$$
	 By Acerbi and Fusco in \cite{AcerbiFusco3}, Lemma 2.1, it was proven that for any $z_1,z_2 \in \mathbb{R}^{nN}$ one has:
	 $$\int_0^1 (1+|z_1+tz_2|^2)^{\frac{p-2}{2}}\, dt \le c(1+|z_1|^2+|z_2|^2)^{\frac{p-2}{2}}.$$
	 In our case, we get:
	 \begin{multline*}\int_{B_r} |V^{\frac{p}{2}}(Dv)|^2\,dx\le \\ \le  c\left[\int_R \left|V^{\frac{q}{2}}(D\psi)\right|^2\,dx+\int_R \left(1+|Dv|^2+|D\psi|^2\right)^{\frac{q-2}{2}}(|Dv|+|D\psi|)\|D\psi|\right]\\ =:c[(I)+(II)].\end{multline*}
	 To estimate $(I)$, we use the obvious property that $$(1+|z_1|^2)^\frac{p}{2}\le 1+(1+|z_1|^2)^\frac{p-2}{2}|z_1|^2$$ and H\"older and Young inequalities, obtaining:
	 \begin{multline*}
	 (I)\le \int_R\left(1+\left|V^{\frac{p}{2}}(D\psi)\right|^2\right)^{\frac{q-p}{p}}\left|V^{\frac{p}{2}}(D\psi)\right|^2\,dx \le \\ \le c \left[\int_R \left| V^{\frac{p}{2}}(D\psi)\right|^2\,dx+\int_R\left|V^{\frac{p}{2}}(D\psi)\right|^{\frac{2q}{p}}\,dx \right] = \\=:c\left[\int_R \left| V^{\frac{p}{2}}(D\psi)\right|^2\,dx+(III)\right]\le c\Delta+(III).
	 \end{multline*}
	 Now, using properties $7$ and $9$ from Lemma \ref{Fonseca1} and Lemma \ref{Fonseca2}, we have
	 \begin{multline*}
	 (III)\le c \int_R \left(T_{\tilde{r},\tilde{s}}\left[\left|V^{\frac{p}{2}}(D[(1-\eta)v])\right|^{\frac{2}{p}}\right]\right)^q\,dx \le \\ \le c(\tilde{s}-\tilde{r})^n\Big(\sup\limits_{t \in (\tilde{r},\tilde{s})}\frac{(\tilde{s}-\tilde{r})^{1-n}}{t-\tilde{r}}\int_{B_t\setminus B_{\tilde{r}}}\left|V^{\frac{p}{2}}(D[(1-\eta)v])\right|^2\,dx+ \\+ \sup\limits_{t \in (\tilde{r},\tilde{s})}\frac{(\tilde{s}-\tilde{r})^{1-n}}{\tilde{s}-t}\int_{B_{\tilde{s}}\setminus B_t}\left|V^{\frac{p}{2}}(D[(1-\eta)v])\right|^2\,dx \Big)^{\frac{q}{p}} \le \\ \le c(s-r)^n\left[(s-r)^{1-n}\left(\sup\limits_{t\in (\tilde{r},\tilde{s})}\frac{\Xi(t)-\Xi(\tilde{r})}{t-\tilde{r}}+\sup\limits_{t\in (\tilde{r},\tilde{s})}\frac{\Xi(\tilde{s})-\Xi(t)}{\tilde{s}-t}\right)\right]^{\frac{q}{p}} \le \\ \le c(s-r)^n\left(\frac{\Delta}{(s-r)^n}\right)^{\frac{q}{p}}
	 \end{multline*}
	 So we have proved that:
	 $$(I)\le c\left[\Delta+(s-r)^n\left(\frac{\Delta}{(s-r)^n}\right)^{\frac{q}{p}}\right]$$
	 To estimate $(II)$, we make repeated use of Young inequality and of the fact that $$(1+|z_1|^2+|z_2|^2)^\frac{p}{2}\le 1+(1+|z_1|^2+|z_2|^2)^\frac{p-2}{2}(|z_1|^2+|z_2|^2)$$
	 In particular
	 \begin{multline*}
	 (II)\le c \Big[ \int_R (1+|Dv|^2+|D\psi|^2)^{\frac{p-2}{2}}(|Dv|+|D\psi|)|D\psi|\,dx+\\+ \int_R(1+|Dv|^2+|D\psi|^2)^{(p-2)\frac{2}{2p}}(|Dv|^2+|D\psi|^2)^\frac{q-p}{p}(|Dv|+|D\psi|)|D\psi|\,dx\Big] \le \\ \le \Big[\int_R\left|V^{\frac{p}{2}}(Dv)\right|^2\,dx+\int_R\left|V^{\frac{p}{2}}(D\psi)\right|^2\,dx+\\+\int_R\left|V^{\frac{p}{2}}(D\psi)\right|^{\frac{2q}{p}}\,dx+\int_R\left|V^{\frac{p}{2}}(Dv)\right|^{\frac{2q}{p}-1}\left|V^{\frac{p}{2}}(D\psi)\right|\,dx\Big]
	 \end{multline*}
	 The first three summands on the right hand side of this inequality are easily controlled by $c\Delta$ or have already been encountered throughout the proof, so that we only need to estimate $$(IV):=\int_R\left|V^{\frac{p}{2}}(Dv)\right|^{\frac{2q}{p}-1}\left|V^{\frac{p}{2}}(D\psi)\right|\,dx.$$
	 Using $q<\frac{3}{2}p$, H\"older inequality yields:
	 $$(IV)\le \left(\int_R \left|V^{\frac{p}{2}}(Dv)\right|^2\,dx\right)^{\frac{2q-p}{2p}}\left(\int_R \left|V^{\frac{p}{2}}(D\psi)\right|^{\frac{2p}{3p-2q}}\right)^{\frac{3p-2q}{2p}}$$
	 The first factor can be easily estimated by $\Delta^{\frac{2q-p}{2p}}$, while using the fact that $q<p+\frac{p}{2n}$, and hence $q<\frac{3}{2}p$, we can estimate the other factor as we did for $(III)$, so that $$(IV)\le c (s-r)^n\left(\frac{\Delta}{(s-r)^n}\right)^{\frac{q}{p}}$$.\\
	 In conclusion, we have 
	 $$\int_{B_r} \left|V^{\frac{p}{2}}(Dv)\right|^2\,dx \le  c\left[\Delta+ (s-r)^n\left(\frac{\Delta}{(s-r)^n}\right)^{\frac{q}{p}}\right]$$
	 and we finish the proof as we did for $p>2$.
\end{proof}
\begin{rmk}[Schmidt's Remark 7.4 in \cite{Schmidt}]\label{remark}
	Let us mention that in the case $q=p$, the inequality \eqref{caccioppoli} holds without the second term on its right-hand side. This can be inferred directly from the proofs. However, in the case $q>p$ we will see that this second term is arbitrarily small. This is the reason why we call \eqref{caccioppoli} a "Caccioppoli inequality".
\end{rmk}
\section{Almost $\mathcal{A}$-harmonicity}
Consider a bilinear form $\mathcal{A}$ on $\mathbb{R}^{nN}$. We assume that the upper bound \begin{equation}\label{limitatezza}|\mathcal{A}|\le \Lambda \end{equation} with $\Lambda>0$ holds and that the Legendre-Hadamard condition \begin{equation}\label{Legendre}\mathcal{A}(\zeta x^T,\zeta x^T)\ge \lambda |x|^2|\zeta|^2 \quad \text{ for all } x\in \mathbb{R}^n, \zeta \in \mathbb{R}^N\end{equation} with ellipticity constant $\lambda>0$ is satisfied.\\ We say that $h\in W^{1,1}_{loc}(\Omega,\mathbb{R}^N)$ is $\mathcal{A}$-harmonic on $\Omega$ iff $$\int_\Omega \mathcal{A}(Dh,D\varphi)\,dx=0$$ holds for all smooth $\varphi:\Omega\to \mathbb{R}^N$ with compact support in $\Omega$.\\
The following two lemmas, whose proof can be found in (\cite{Schmidt}, Lemma 7.8, 7.7 and 6.8) will enable us to approximate $W^{1,p}$-minimizers with functions that are $\mathcal{A}$-harmonic. 
\begin{lem} \label{lemmanove}
	Let $f$ satisfy $(A.1)$ and $(A.3')$ for a given $M>0$. Choose any $M>0$. Then, for any given $z$ such that $|z|>M$, we have that $\mathcal{A}=D^2f(z)$ satisfies the Legendre-Hadamard condition $$\mathcal{A}(\zeta x^T,\zeta x^T)\ge \lambda |x|^2|\zeta|^2 \quad \text{ for all } x \in \mathbb{R}^n\text{ and } \zeta \in \mathbb{R}^N$$ with ellipticity constant $\lambda=2\gamma$.
	\begin{proof}
		Let $u$ be the affine function $u(x)=zx$ with $z$ such that $|z|>M$. Quasiconvexity in $z$ ensures that $u$ is a $W^{1,p}$-minimizer of the functional $\mathcal{F}$ induced by $f$ and that the function:
		$$G_\varphi(t)=\mathcal{F}_{|B_1}(u+t\varphi)-\gamma\int_{B_1} (1+|tD\varphi|^2)^{\frac{p}{2}-1}|tD\varphi|^2 \, dx$$ has a minimum in $t=0$ for any $\varphi \in W_0^{1,p}(B_1,\mathbb{R}^N)$ and, in the same way as it is done in (\cite{Giusti}, Prop. 5.2), from $G'(0)=0$ and $G''(0)\ge0$ the Legendre-Hadamard condition will follow.\\
		As a matter of fact, from $G''(0)\ge 0$, we obtain:
		\begin{equation}\label{dasommare}\int_{B_1} \frac{\partial^2 F}{\partial z_k^\alpha\partial z_j^\beta}(z_0)D_k\varphi^\alpha D_j\varphi^\beta\,dx\ge 2\gamma \int_{B_1} |D\varphi^2| \,dx\end{equation}
		for every $\varphi \in C_c^1(B_1,\mathbb{R}^N)$.
		Let us $\varphi=\lambda+i\mu$ and write \eqref{dasommare} for $\lambda$ and for $\mu$, i.e.:
		\begin{equation}\int_{B_1} \frac{\partial^2 F}{\partial z_k^\alpha\partial z_j^\beta}(z_0)D_k\lambda^\alpha D_j\lambda^\beta\,dx\ge 2\gamma \int_{B_1} |D\lambda^2| \,dx\end{equation}
		and
		\begin{equation}\int_{B_1} \frac{\partial^2 F}{\partial z_k^\alpha\partial z_j^\beta}(z_0)D_k\mu^\alpha D_j\mu^\beta\,dx\ge 2\gamma \int_{B_1} |D\mu^2|\,dx\end{equation}	
		we obtain:
		\begin{equation}\int_{B_1} \frac{\partial^2 F}{\partial z_k^\alpha\partial z_j^\beta}(z_0)\left[D_k\lambda^\alpha D_j\lambda^\beta+D_k\mu^\alpha D_j\mu^\beta\right]\,dx\ge 2\gamma \int_{B_1} |D\lambda^2|+|D\mu^2|\,dx\end{equation}
		and hence:
		$$\text{Re}\int_{B_1}\frac{\partial^2 F}{\partial z_k^\alpha \partial z_j^\beta}(z_0)D_k\varphi^\alpha D_j\overline{\varphi}^\beta\,dx \ge 2\gamma \int_{B_1}|D\varphi|^2 \,dx$$
		Now, consider any $\xi \in \mathbb{R}^n$, $\eta \in \mathbb{R}^N$, $\tau \in \mathbb{R}$ and $\varPsi(x)\in C_c^\infty(B_1,\mathbb{R})$ and take $\varphi$ to be $\varphi(x)=\eta e^{i\tau(\xi \cdot x)}\varPsi(x)$. Since $\varphi^\alpha(x)=\eta^\alpha\varPsi(x)e^{i\tau \xi \cdot x}$, we have
		$$\int_{B_1}\frac{\partial^2 F}{\partial z_k^\alpha \partial z_j^\beta}(z_0)\eta^\alpha\eta^\beta[\tau^2\xi_k\xi_j\varPsi^2+D_k\varPsi D_j\varPsi]\,dx \ge 2\gamma|\eta|^2\int_{B_1}(|D\varPsi|^2+\tau^2|\xi|^2|\varPsi(x)|^2)\,dx.$$
		Dividing by $\tau^2$ and letting $\tau \to \infty$ we get:
		$$\int_{B_1}\frac{\partial^2 F}{\partial z_k^\alpha \partial z_j^\beta}(z_0)\xi_k\xi_j\eta^\alpha\eta^\beta\varPsi^2(x)\,dx\ge 2\gamma|\eta|^2|\xi|^2\int_{B_1}\varPsi^2(x)\,dx $$
		and since this holds for all $\varPsi \in C_c^\infty(B_1,\mathbb{R})$ the proposition is proved. 
	\end{proof}
\end{lem}
\begin{rmk}
	Assume $f \in C^2_{\text{loc}}(\mathbb{R}^{nN})$. Then for each $L>0$, there is a modulus of continuity $\omega_L:[0,+\infty[\to [0,+\infty[$ satisfying $\lim\limits_{z \to 0} \omega_L(z)=0$ such that for all $z_1,z_2 \in \mathbb{R}^{nN}$ we have:
	$$|z_1|\le L, \ |z_2|\le L+1 \Rightarrow |D^2f(z_1)-D^2f(z_2)|\le \omega_L(|z_1-z_2|^2).$$
	Moreover, $\omega_L$ can be chosen such that the following properties hold:
	\begin{enumerate}
		\item $\omega_L$ is non-decreasing,
		\item $\omega^2_L$ is concave,
		\item $\omega^2_L(z)\ge z$ for all $z \ge 0$.
	\end{enumerate}
\end{rmk}
\begin{lem} \label{dieci}
	Let $f$ satisfy $(A.1)$,$(A,2)$,$(A.3)$ for a given $M>0$. Choose any $L>M>0$ and take $u \in W^{1,p}$ to be a $W^{1,p}$-minimizer of $\mathcal{F}$ on some ball $B_\rho(x_0)$, where $q\le p+1$. Then for all $z:\ M<|z|\le L$ and $\varphi \in C_c^\infty(B_\rho(x_0))$ we have
	\begin{equation} \label{questarobaqua}
	\left| \fint_{B_\rho(x_0)} D^2f(z)(Du-z,D\varphi)\,dx \right|\le c\sqrt{\Phi_p}\omega_{L}(\Phi_p)\sup\limits_{B_\rho(x_0)}|D\varphi|.
	\end{equation}
	where $\Phi_p:=\Phi_p(u,x_0,\rho,z)$, the constant $c$ depends only on $n$,$N$,$p$,$q$,$\Gamma$,$L$ and $\omega_{L}$ is the abovementioned modulus of continuity (see also \cite{Schmidt}).
\end{lem}
\begin{proof}
The proof of a similar result, in \cite{Schmidt}, will be adapted and explicitely repeated for the convenience of the reader.\\
We may assume, without loss of generality, that $x_0=0$ and $\sup\limits_{B_\rho}|D\varphi|=1$. Setting $v(x):=u(x)-zx$, the Euler equation of $F$ gives \begin{equation}\label{undici}\left|\fint_{B_\rho} D^2f(z)(Dv,D\varphi)\,dx \right|\le \fint_{B_\rho} \left| D^2f(z)(Dv,D\varphi)+Df(z)D\varphi-Df(Du)D\varphi\right|\,dx\end{equation}
Now we estimate the integrand on the right-hand side.\\ On the set $\{x\in B_\rho:|Dv|\le 1\}$ we have $|Dv|^2\le 2\left|V^{\frac{p}{2}}(Dv)\right|^2$. Using this, Remark 2 and the concavity of $\omega_L$ we have: 
\begin{multline} \label{dodici}
\left| D^2f(z)(Dv,D\varphi)+Df(z)D\varphi-Df(Du)D\varphi\right| \\ \le \int_0^1 |D^2f(z)-D^2f(z+tDv)|\,dt|Dv| \\ \le \omega_L(|Dv|^2)|Dv|\le c\omega_L\left(\left|V^{\frac{p}{2}}(Dv)\right|^2\right)\left|V^{\frac{p}{2}}(Dv)\right|.
\end{multline} 
On the set $\{x\in B_\rho:|Dv|\ge 1\}$, Lemma \ref{Acerbo} implies
\begin{multline}\label{tredici}
\left| D^2f(z)(Dv,D\varphi)+Df(z)D\varphi-Df(Du)D\varphi\right|\le \\ \le \sup\limits_{|z|\le L+2} |D^2f(z)||Dv|+c|V^{q-1}(Dv)|\le c|Dv|^{\max\{q-1,1\}}\le\\ \le c\left|V^{\frac{p}{2}}(Dv)\right|^2.
\end{multline}
Combining \eqref{undici}, \eqref{dodici} and \eqref{tredici} and noticing that property $3$ of $\omega_L$ stated in Remark $2$ implies that $$\max\left\{\omega_L\left(\left|V^{\frac{p}{2}}(Dv)\right|^2\right)\left|V^{\frac{p}{2}}(Dv)\right|,\left|V^{\frac{p}{2}}(Dv)\right|^2\right\}=\omega_L\left(\left|V^{\frac{p}{2}}(Dv)\right|^2\right)\left|V^{\frac{p}{2}}(Dv)\right|$$ we have
$$\left| \fint_{B_\rho} D^2f(z)(Dv,D\varphi)\,dx\right|\le c\fint_{B_\rho}\omega_L\left(\left|V^{\frac{p}{2}}(Dv)\right|^2\right)\left|V^{\frac{p}{2}}(Dv)\right|\,dx.$$
Now we apply H\"older to obtain
$$\left| \fint_{B_\rho} D^2f(z)(Dv,D\varphi)\,dx\right|\le c \left[\fint_{B_\rho}\omega^2_L\left(\left|V^{\frac{p}{2}}(Dv)\right|^2\right)\right]^{\frac{1}{2}}\left[\fint_{B_\rho} \left|V^{\frac{p}{2}}(Dv) \right|^2\right]^{\frac{1}{2}}$$
Now, by Jensen, using the concavity of $\omega_L^2$, to obtain 
$$\left|\fint_{B_\rho}D^2f(z)(Dv,D\varphi)\,dx\right|\le c\sqrt{\Phi_p}\omega_L(\Phi_p).$$
This completes the proof.

\end{proof}
\begin{lem}\label{Armonic}
	Fix $1<p<\infty$, $0<\lambda \le \Lambda< \infty$ and $\varepsilon>0$. Then there is a $\delta(n,N,p,\Lambda,\lambda,\varepsilon)>0$ such that the following assertion holds:\\
	For all $s\in (0,1]$, for all $\mathcal{A}$ satisfying \eqref{limitatezza} and \eqref{Legendre} and for each $u \in W^{1,p}(B_\rho(x_0);\mathbb{R}^N)$ with:
	\begin{equation*}
	\fint_{B_{\rho}(x_0)} |V^{\frac{p}{2}}(Du)|^2\,dx \le s^2
	\end{equation*}
	and
	\begin{equation*}
	\left| \fint_{B_\rho(x_0)} \mathcal{A}(Du,D\varphi)\,dx\right|\le s\delta\sup\limits_{B_\rho(x_0)}|D\varphi|
	\end{equation*}
	for all smooth $\varphi:B_\rho(x_0)\to \mathbb{R}^N$ with compact support in $B_\rho(x_0)$ there is an $\mathcal{A}$-harmonic function $h \in C^{\infty}_{loc}(B_\rho(x_0),\mathbb{R}^N)$ with 
	\begin{equation*}
	\sup\limits_{B_{\rho/2}(x_0)}|Dh|+\rho\sup\limits_{B_{\rho/2}(x_0)}|D^2h|\le c
	\end{equation*}
	and
	\begin{equation*}
	\fint_{B_{\rho/2}(x_0)}\left|V^{\frac{p}{2}}\left(\frac{u-sh}{\rho}\right) \right|^2\,dx \le s^2\varepsilon.
	\end{equation*}
	Here $c$ denotes a constant depending only on $n,N,p,\Lambda,\lambda$.
	
\end{lem}
\medskip \medskip
\section{Excess decay estimate}
\begin{prop}
	Let $z_0$ be s.t. $|z_0|>M+1$ and $x_0$ be s.t. $$\lim\limits_{\rho\to 0}\fint_{B_\rho(x_0)}\left|V^{\frac{p}{2}}(Du(x)-z_0)\right|^2=0$$ then $$\Phi_p(u,x_0,\rho)\to 0 \quad \text{ as } \quad \rho\to 0.$$\\
	\begin{proof} Let $(Du)_\rho:=\fint_{B_\rho(x_0)}|Du|$. We have, using \eqref{prop1}, \eqref{comparability} and convexity of $\left|W^{\frac{p}{2}}(z)\right|^2$:
	\begin{multline*}
	\Phi_p(u,x_0,\rho) =\fint_{B_\rho(z_0)}|V^{\frac{p}{2}}[Du-(Du)_\rho]|^2\,dx \le \\ \le c\left[ \fint_{B_\rho(z_0)}|V^{\frac{p}{2}}[Du-z_0]|^2\,dx+|V^{\frac{p}{2}}[z_0-(Du)_\rho]|^2 \right] \le \\ \le c\fint_{B_\rho(z_0)}\left|V^{\frac{p}{2}}[Du-z_0]\right|^2\,dx \to 0
	\end{multline*}
	\end{proof}
\end{prop}
Finally, we can prove
\begin{lem}
	Assume $q$ and $p$ are real numbers such that $q<p+\frac{\min\{2,p\}}{2n}$.\\
	Let $f$ satisfy assumptions $(A.1)$, $(A.2)$ and $(A.3)$ for a given $M>0$.\\
	Choose any $L>M+1>0$, $\alpha \in (0,1)$, $z_0 \in \mathbb{R}^{nN}$ such that $|z_0|>M+1$.\\
	Then there are constants $\varepsilon_0>0$, $\theta \in (0,1)$ and a radius $\rho^*>0$ depending on $n,N,L,p,q,\Gamma,\alpha,\gamma,x_0,z_0$ and $\Lambda_L:=\max\limits_{B_{L+2}}|D^2f|$ and with $\varepsilon_0$ depending additionally on $\omega_L$ such that the following holds.\\
	Consider $u$ a $W^{1,p}$-minimizer of $\mathcal{F}$ on $B_\rho(x_0)$, with $\rho<\rho^*$ and $x_0 \in \mathbb{R}^n$ satisfying $$\lim\limits_{\rho\to 0 }\fint_{B_\rho(x_0)} \left|V^{\frac{p}{2}}(Du(x)-z_0)\right|^2=0.$$If the following conditions hold
	\begin{equation}\label{unouno}\Phi_p(u,x_0,\rho)\le\varepsilon_0\end{equation} and \begin{equation}\label{duedue}|(Du)_{x_0,\rho}|\le L\end{equation} 
	then $$\Phi_p(u,x_0,\theta\rho)\le \theta^{2\alpha}\Phi_p(u,x_0,\rho).$$ 
\end{lem}
\begin{proof}
	Let $z_0$ be such that $|z_0|>M+1$ and $x_0$ any point such that $\lim\limits_{\rho\to 0 }\fint_{B_\rho(x_0)} |Du(x)-z_0|^p=0$. In what follows, for simplicity of notation, we assume that $x_0=0$ and we abbreviate $$z=(Du)_\rho:=\fint_{B_\rho}Du\,dx$$ and $$\Phi_p(\cdot):=\Phi_p(u,0,\cdot).$$ where $\rho>0$ is any positive value small enough (smaller than a $\rho^*$ that will be determined throughout the proof).\\ Since the claim is obvious in the case $\Phi_p(\rho)=0$ we can assume $\Phi_p(\rho)\not=0$.\\
	Setting $$w(x):=u(x)-zx\quad \text{ and } \quad s:=\sqrt{\Phi_p(\rho)}$$ we have by definition of $\Phi_p(\rho)$, $$\fint_{B_\rho}|V^{\frac{p}{2}}(Dw)|^2\,dx=s^2=\Phi_p(\rho).$$ Next we will approximate by $\mathcal{A}$--harmonic functions, where $\mathcal{A}:=D^2f(z)$.\\
	If we choose $\rho<\rho^*:=r_1(z_0)$ as in Lemma \ref{puntiregolari}, we have $|z|>M+1$, hence, from $|\mathcal{A}|\le \max\limits_{B_{L+2}}|D^2f|=:\Lambda_L$ and Lemma \ref{lemmanove} we deduce that $\mathcal{A}$ satisfies \eqref{limitatezza} with a bound $\Lambda_L$ and \eqref{Legendre} with ellipticity constant $2\gamma$. Lemma \ref{dieci} yields the estimate:
	\begin{equation*}
	\left| \fint_{B_\rho}\mathcal{A}(D\omega,D\varphi)\,dx\right| \le sC_2\omega_{L}\left(\Phi_p(\rho)\right)\sup\limits_{B_\rho}|D\varphi|
	\end{equation*}
	for all $\rho<\rho^*$ and for all smooth functions $\varphi:B_{\rho}\to \mathbb{R}^N$ with compact support in $B_\rho$, where $C_2$ is a positive constant depending on $n,N,p,q,\Gamma,L,\Lambda_L$.\\ For $\varepsilon>0$ to be specified later, we fix the corresponding constant $\delta(n,N,p,\Lambda_L,\gamma,\varepsilon)>0$ from Lemma \ref{Armonic}.\\ Now, let $\varepsilon_0=\varepsilon_0(n,N,p,\Lambda_L,\gamma,\varepsilon)$ be small enough so that \eqref{unouno} implies:
	\begin{equation} \label{settepuntoquindici}
	C_2\omega_{L}(\Phi_p(\rho))\le \delta
	\end{equation}
	\begin{equation}\label{settepuntosedici}
	s=\sqrt{\Phi_p(\rho)}\le 1.
	\end{equation}
	We apply Lemma \ref{Armonic}. The lemma ensures the existence of an $\mathcal{A}$-harmonic function $h\in C^{\infty}_{loc}(B_\rho;\mathbb{R}^N)$ such that
	\begin{equation*}
	\sup\limits_{B_{\rho/2}}|Dh|+\rho\sup\limits_{B_{\rho/2}}|D^2h|\le c
	\end{equation*}
	where $c=c(n,N,p,\Lambda_L,\gamma)$ and
	\begin{equation}
	\fint_{B_{\rho/2}} \left| V^{\frac{p}{2}}\left(\frac{w-sh}{\rho}\right)\right|^2\,dx\le s^2\varepsilon.
	\end{equation}
	Now fix $\theta\in (0,1/4]$. Taylor expansion implies the estimate:
	\begin{equation*}
	\sup\limits_{x \in B_{2\theta\rho}}|h(x)-h(0)-Dh(0)x|\le \frac{1}{2}(2\theta\rho)^2 \sup\limits_{x \in B_{\rho/2}} |D^2h| \le c\theta^2\rho.
	\end{equation*}
	Using \eqref{prop1} and \eqref{prop2} together with what we have obtained we get:
	\begin{multline*}
	\fint_{B_{2\theta\rho}}\left|V^{\frac{p}{2}}\left(\frac{w(x)-sh(0)-sDh(0)x}{2\theta\rho}\right)\right|^2\,dx \le \\ \le c\Big[\theta^{-n-\max\{2,p\}}\fint_{B_{\rho/2}}\left|V^{\frac{p}{2}}\left(\frac{w-sh}{\rho}\right)\right|^2\,dx+\\+ \fint_{B_{2\theta\rho}} \left|V^{\frac{p}{2}}\left(s\frac{h(x)-h(0)-Dh(0)x}{2\theta\rho}\right)\right|^2\,dx\Big]\le\\ \le c\left[\theta^{-n-\max\{2,p\}}s^2\varepsilon+\left|V^{\frac{p}{2}}(\theta s)\right|^2\right]\le \\ \le c\left[\theta^{-n-\max\{2,p\}}s^2\varepsilon+\theta^{2}s^2\right]
	\end{multline*}
	Setting $\varepsilon:=\varepsilon(\theta)=\theta^{n+2+\max\{2,p\}}$ (so, remember that $\varepsilon$ and hence $\delta$ and $\varepsilon_0$ depend on whatever $\theta$ we wish to choose) and recalling the definitions of $w$ and $s$ we have: \begin{equation}\label{ciccio}
	\fint_{B_{2\theta\rho}} \left| V^{\frac{p}{2}}\left(\frac{u(x)-zx-s(h(0)+Dh(0)x)}{2\theta\rho}\right)\right|^2\,dx\le c\theta^2\Phi_p(\rho).
	\end{equation}
	On the other hand, we remark that, using the definition of $s$ and properties of $h$:
	\begin{equation} \label{settepuntodiciannove}
	|sDh(0)|^2\le c^2\Phi_p(\rho)
	\end{equation}
	We can take $\varepsilon_0$ small enough such that \eqref{unouno} implies also: $$s\le \frac{1}{c}$$
	and that would imply \begin{equation}\label{settepuntoventidue}
	|sDh(0)|^2\le 1.
	\end{equation}\\
	Using this fact together with \eqref{settepuntodiciannove} and  \eqref{prop1} we get 
	\begin{multline} \label{ciccio2}
	\Phi_p(2\theta\rho,z+sDh(0))\le\\ \le c\left[(2\theta)^{-n}\left(\fint_{B_\rho}|V^{\frac{p}{2}}(Du-z)|^2\,dx+|V^{\frac{p}{2}}(sDh(0))|^2\right)\right]\le\\ \le c\left[\theta^{-n}\left(\Phi_p(\rho)+|sDh(0)|^2\right)\right] \le c\theta^{-n}\Phi_p(\rho).
	\end{multline}
	Now we need to use \eqref{caccioppoli} with $\zeta=sh(0)$ and $z+sDh(0)$ instead of $z$, and we can be sure that $|z+sDh(0)|>M$ because $|sDh(0)|\le 1$.\\ 
	Now, we can combine \eqref{ciccio} and \eqref{ciccio2} and Caccioppoli inequality \eqref{caccioppoli} with $\zeta=sh(0)$ and $z+sDh(0)$ instead of $z$, and we get 
	\begin{equation} \label{settepuntoventuno}
	\Phi_p(\theta\rho,z+sDh(0))\le c\left[\theta^2\Phi_p(\rho)+\theta^\frac{2q}{p}\Phi_p(\rho)^{\frac{q}{p}}+\theta^{-n\frac{q}{p}}\Phi_p(\rho)^{\frac{q}{p}}\right].	\end{equation}
	Thereby the condition $|z+sDh(0)|\le L+1$ of Lemma \ref{caccioppolilemma} can be deduced from \eqref{settepuntoventidue}.\\
	Now, if $\varepsilon_0$ is chosen small enough, depending on $\theta$, \eqref{unouno} implies the following:
 	\begin{equation}\label{settepuntoventitre}
	\theta^{-n\frac{q}{p}}\Phi_p(\rho)^{\frac{q-p}{p}}\le \theta^2,
	\end{equation}
	and from the fact that $\theta\le 1$ we have
	$$\Phi_p(\theta\rho,z+sDh(0))\le c\theta^2\Phi_p(\rho).$$
	For $q=p$, however, the last inequality holds without further assumptions since the last term on the right hand side of \eqref{settepuntoventuno} does not occur (see Remark \ref{remark}).\\
	Using Lemma 6.2 in \cite{Schmidt} (written in the same notation as ours except for $A$ instead of $z$) we deduce from the previous inequality:
	\begin{equation}\label{settepuntoventiquattro}
	\Phi_p(\theta\rho)\le C_3\theta^2\Phi_p(\rho),
	\end{equation}
	where $C_3>0$ depends on $n,N,p,q,\Gamma,\gamma,\Lambda_L,L$.\\ Finally, we choose $\theta \in (0,\frac{1}{4}]$ (depending on $\alpha$ and whatever $C_3$ depends on) small enough such that \begin{equation}\label{settepuntoventicinque}
	C_3\theta^2\le \theta^{2\alpha}
	\end{equation}
	holds, and $\varepsilon_0$ small enough such that \eqref{settepuntoquindici}, \eqref{settepuntosedici}, \eqref{settepuntoventidue}, \eqref{settepuntoventitre} follow from \eqref{unouno}. Taking into account \eqref{settepuntoventiquattro} and \eqref{settepuntoventicinque} the proof of the proposition is complete.
\end{proof}
The following adaptation of (\cite{Schmidt}, Lemma 7.10) is then a trivial consequence of this last lemma.
\begin{lem} \label{iterato}
	Assume $q$ and $p$ are real numbers such that $q<p+\frac{\min\{2,p\}}{2n}$.\\
	Let $f$ satisfy assumptions $(A.1)$, $(A.2)$ and $(A.3)$ for a given $M>0$.\\
	Choose any $L>2M+2>0$, $\alpha \in (0,1)$, $z_0 \in \mathbb{R}^{nN}$ such that $|z_0|>M+1$.\\
	Then there is a constant $\tilde{\varepsilon}_0>0$ and a radius $\rho^*>0$ depending on $n,N,L,p,q,\Gamma,\alpha,\gamma,x_0,z_0$ and $\Lambda_L:=\max\limits_{B_{L+2}}|D^2f|$ and with $\tilde{\varepsilon}_0$ depending additionally on $\omega_L$ such that the following holds.\\
	Consider $u$ a $W^{1,p}$-minimizer of $\mathcal{F}$ on $B_\rho(x_0)$, with $\rho<\rho^*$ and $x_0 \in \mathbb{R}^n$ satisfying $$\lim\limits_{\rho\to 0 }\fint_{B_\rho(x_0)} |Du(x)-z_0|^p=0.$$If the following conditions hold
	\begin{equation}\label{unounouno}\Phi_p(u,x_0,\rho)\le\tilde{\varepsilon}_0\end{equation} and \begin{equation}\label{dueduedue}|(Du)_{x_0,\rho}|\le \frac{L}{2}\end{equation} 
	then there is a constant $c$ depending on $n,N,L,p,q,\Gamma,\alpha,\gamma,x_0,z_0$ such that $$\Phi_p(u,x_0,r)\le c\left(\frac{r}{\rho}\right)^{2\alpha}\Phi_p(u,x_0,\rho)$$ for any $r<\rho$. 
\end{lem}

\noindent
\section*{Regularity}
Now we are able to prove our main result 
\begin{proof}[Proof of Theorem \ref{maintheorem}]
Let $x_0$ be such that $\exists z_0: |z_0|>M+1$ with the property that $$\fint_{B_\rho(x_0)}|Du-z_0|^p \to 0 \text{ as } \rho\to 0.$$ and choose any~$\alpha\in (0,1)$ and $L=4|z_0|>M$.\\ Then there is $r_2>0$ small enough such that $|\fint_{B_\rho(x_0)} Du|<L/2$ for all $\rho<r_2$ and, since $\Phi_p(u,x_0,\rho)\to 0$ as $\rho\to 0^+$, there is $r_3$ such that, for all $\rho<r_3$, $\Phi_p(u,x_0,\rho)<\varepsilon_0$.\\ Applying lemma \ref{iterato} we have that $Du$ belongs to the Morrey-Campanato space $L^{\lambda,2}$ with $\lambda=2\alpha+n>n$ so that, because of the continuous immersion \cite{Campanato} $L^{\lambda,2}\hookrightarrow C^{0,\frac{2\alpha+n-n}{2}}=C^{0,\alpha}$ we obtain that $Du \in C^{0,\alpha}$ and so $u \in C^{1,\alpha}(B_\rho(x_0))$ choosing $\rho<\min\{\rho^*,r_2,r_3\}$. So $x_0\in \Reg(u)$.\\
Of course $\Reg(u)$ is an open set by definition.\\
We will now argue by contradiction to prove that it is dense. Assume there is a point $x \in \Omega$ and a radius $r>0$ such that $B_r(x)$ is entirely outside $\Reg(u)$. Since $Du \in L^p \subseteq L^1$, by Lebesgue-Besicovitch theorem, for almost all points $y$ in $B_r(x)$ this would mean that $\lim\limits_{r \to 0} \fint_{B_r(y)} |Du| < M$ and so $|Du|$ is essentially bounded by $M$ in $B_r(x)$, which contradicts the hypothesis that $B_r(x)$ is outside $\Reg(u)$.
\end{proof}

\end{document}